\documentclass[11pt]{amsart}
\usepackage[margin=30mm]{geometry}
\usepackage{amsmath,amssymb}
\usepackage{amsthm}
\usepackage{mathrsfs}

\newtheorem{thm}{Theorem}[section]

\newtheorem{lem}{Lemma}[section]

\newtheorem{defi}{Definition}[section]
\newtheorem{ex}{Example}[section]
\newtheorem{rem}{Remark}[section]

\newtheorem*{lemma}{\it{Lemma}}

\usepackage{enumitem}

\begin{document}

\title{On non-chaotic hyperbolic sets}
\author{Noriaki Kawaguchi}
\address{Department of Mathematical and Computing Science, School of Computing, Institute of Science Tokyo, 2-12-1 Ookayama, Meguro-ku, Tokyo 152-8552, Japan}
\email{gknoriaki@gmail.com}

\begin{abstract}
We give necessary and sufficient conditions for a hyperbolic set to be non-chaotic (or, conversely, chaotic) in a certain sense.
\end{abstract}

\maketitle

\markboth{NORIAKI KAWAGUCHI}{On non-chaotic hyperbolic sets}

\section{Introduction}

{\em Hyperbolic sets} are important objects in the study of differentiable dynamical systems (see \cite{KH} for background). In particular, their existence is closely tied to the chaotic behavior of such systems. Hyperbolic sets are typically surrounded by many hyperbolic periodic points and their transversal homoclinic points, where rich chaotic dynamics occur. In this paper, however, we focus on situations in which a hyperbolic set degenerates and fails to produce chaotic behavior in its neighborhood. Our aim is to characterize precisely those hyperbolic sets that are non-chaotic (or, conversely, chaotic) in a certain sense.  Note that this issue was studied by, e.g., Anosov \cite{A1}. In \cite{K}, the author gave several necessary and sufficient conditions for a locally maximal hyperbolic set to be degenerate and thus non-chaotic. This paper deals with a hyperbolic set that is not necessarily locally maximal.
 
We begin by defining chain recurrence and chain components. Throughout, $X$ denotes a compact metric space endowed with a metric
\[
d\colon X\times X\to[0,\infty).
\]
Let $f\colon X\to X$ be a continuous map. For $\delta>0$, a sequence $(x_i)_{i=0}^k$ of points in $X$, where $k\ge1$, is called a {\em $\delta$-chain} of $f$ if
\[
\sup_{0\le i\le k-1}d(f(x_i),x_{i+1})\le\delta.
\]
For $x,y\in X$, the notation $x\rightarrow y$ means that for any $\delta>0$, there is a $\delta$-chain $(x_i)_{i=0}^k$ of $f$ with $x_0=x$ and $x_k=y$. We say that $f$ is {\em chain transitive} if $x\rightarrow y$ for all $x,y\in X$. An $x\in X$ is called a {\em chain recurrent point} for $f$ if $x\rightarrow x$. We denote by $CR(f)$ the set of chain recurrent points for $f$. Note that $CR(f)$ is a closed subset of $X$ and satisfies $f(CR(f))=CR(f)$. A relation $\leftrightarrow$ in
\[
CR(f)^2=CR(f)\times CR(f)
\]
is defined by: for any $x,y\in X$, $x\leftrightarrow y$ if and only if $x\rightarrow y$ and $y\rightarrow x$. Note that $\leftrightarrow$ is a closed equivalence relation in $CR(f)^2$ and satisfies $x\leftrightarrow f(x)$ for all $x\in X$. An equivalence class of $\leftrightarrow$ is called a {\em chain component} for $f$. Let $\mathcal{C}(f)$ denote the set of chain components for $f$. Note that  
\begin{itemize}
\item $CR(f)=\bigsqcup_{C\in\mathcal{C}(f)}C$, a disjoint union,
\item every $C\in\mathcal{C}(f)$ is a closed subset of $CR(f)$ and satisfies $f(C)=C$,
\item $f|_C\colon C\to C$ is chain transitive for all $C\in\mathcal{C}(f)$.
\end{itemize}

{\em Expansiveness} is a topological feature of dynamics on hyperbolic sets. Its basic definition is as follows. 

\begin{defi}
\normalfont
A homeomorphism $f\colon X\to X$ is said to be {\em expansive} if there is $e>0$ such that
\[
\sup_{i\in\mathbb{Z}}d(f^i(x),f^i(y))\le e
\]
implies $x=y$ for all $x,y\in X$. In this case, $e>0$ is called an expansive constant for $f$.
\end{defi}

{\em Sensitivity} is a characteristic feature of chaotic dynamical systems. Let $f\colon X\to X$ be a continuous map and let $x\in X$. For $a>0$, $x$ is called an {\em $a$-sensitive point} for $f$ if for any $\delta>0$, there is $y\in X$ such that $d(x,y)\le\delta$ and
\[
d(f^i(x),f^i(y))>a
\]
for some $i\ge0$. Let $Sen_a(f)$ denote the set of $a$-sensitive points for $f$. We denote by $Sen(f)$ the set of {\em sensitive points} for $f$:
\[
Sen(f)=\bigcup_{a>0}Sen_a(f).
\] 
We say that $f$ is {\em equicontinuous} if for any $\epsilon>0$, there is $\delta>0$ such that $d(z,w)\le\delta$ implies
\[
\sup_{i\ge0}d(f^i(z),f^i(w))\le\epsilon
\]
for all $z,w\in X$. We easily see that $f$ is equicontinuous if and only if $Sen(f)=\emptyset$. It is not difficult to prove that if $f$ is surjective and equicontinuous, then $f$ is a homeomorphism and $f^{-1}$ is also equicontinuous.

Given a continuous map $f\colon X\to X$, if $Sen(f|_{CR(f)})\ne\emptyset$, then $CR(f)$ is obviously an infinite set. If $Sen(f|_{CR(f)})=\emptyset$, then $f|_{CR(f)}$ is an equicontinuous homeomorphism; therefore, whenever $f$ is an expansive homeomorphism, $Sen(f|_{CR(f)})=\emptyset$ implies that $CR(f)$ is a finite set. In other words, we obtain the following lemma.

\begin{lem}
For any expansive homeomorphism $f\colon X\to X$, the following conditions are equivalent
\begin{itemize}
\item $Sen(f|_{CR(f)})\ne\emptyset$,
\item $CR(f)$ is an infinite set.
\end{itemize}
\end{lem}

A proof of the following lemma can be found in \cite{BS} (see \cite[Corollary 2.4.2]{BS}).

\begin{lemma}
Let $f\colon X\to X$ be an expansive homeomorphism. If $X$ is an infinite set, then there are $x,y\in X$ such that $x\ne y$ and
\[
\lim_{i\to\infty}d(f^i(x),f^i(y))=0.
\]
\end{lemma}

Given a continuous $f\colon X\to X$, we denote by $Per(f)$ the set of {\em periodic points} for $f$:
\[
Per(f)=\bigcup_{j\ge1}\{x\in X\colon f^j(x)=x\}.
\]
By using the above lemma, we shall prove the following.

\begin{lem}
For any expansive homeomorphism $f\colon X\to X$, the following conditions are equivalent
\begin{itemize}
\item[(1)] $CR(f)$ is a finite set,
\item[(2)] every $C\in\mathcal{C}(f)$ is a finite set,
\item[(3)] $CR(f)=Per(f)$.
\end{itemize}
\end{lem}

\begin{proof}
The implications (1)$\implies$(2)$\implies$(3) hold for any continuous map $f\colon X\to X$. If $CR(f)$ is an infinite set, then as $f$ is an expansive homeomorphism and so is $f|_{CR(f)}$, by the above lemma, we have
\[
\lim_{i\to\infty}d(f^i(x),f^i(y))=0
\]
for some $x,y\in CR(f)$ with $x\ne y$. It follows that $x\not\in Per(f)$ or $y\notin Per(f)$; therefore, $CR(f)\setminus Per(f)\ne\emptyset$. This proves the implication (3)$\implies$(1) and thus the lemma has been proved.
\end{proof}

Combining Lemmas 1.1 and 1.2, we obtain the following theorem.

\begin{thm}
For any expansive homeomorphism $f\colon X\to X$, the following conditions are equivalent
\begin{itemize}
\item[(0)] $Sen(f|_{CR(f)})=\emptyset$,
\item[(1)] $CR(f)$ is a finite set,
\item[(2)] every $C\in\mathcal{C}(f)$ is a finite set,
\item[(3)] $CR(f)=Per(f)$.
\end{itemize}
\end{thm}

Let $f\colon X\to X$ be a homeomorphism. For $x\in X$, we define two subsets $W^u(x)$, $W^s(x)$ of $X$ by
\[
W^u(x)=\{y\in X\colon\lim_{i\to-\infty}d(f^i(x),f^i(y))=0\},
\]
\[
W^s(x)=\{y\in X\colon\lim_{i\to+\infty}d(f^i(x),f^i(y))=0\}.
\]

\begin{rem}
\normalfont
For an expansive homeomorphism $f\colon X\to X$, if $CR(f)$ is a finite set, then $X$ is a countable set. This can be seen as follows. If $CR(f)$ is a finite set, then, $CR(f)=Per(f)$. It follows that
\[
X=\bigcup_{p,q\in Per(f)}\left[W^u(p)\cap W^s(q)\right].
\]
Since $f$ is expansive, $W^u(p)\cap W^s(q)$ is a  countable set for all $p,q\in Per(f)$; therefore, $X$ is a countable set.
\end{rem}

{\em Shadowing} is another topological feature of dynamics on hyperbolic sets. Let us define shadowing on a subset.

\begin{defi}
\normalfont
Let $f\colon X\to X$ be a homeomorphism and let $\xi=(x_i)_{i\in\mathbb{Z}}$ be a sequence of points in $X$.
\begin{itemize}
\item For $\delta>0$, $\xi$ is called a {\em $\delta$-pseudo orbit} of $f$ if
\[
\sup_{i\in\mathbb{Z}}d(f(x_i),x_{i+1})\le\delta.
\]
\item For $x\in X$ and $\epsilon>0$, $\xi$ is said to be {\em $\epsilon$-shadowed} by $x$ if
\[
\sup_{i\in\mathbb{Z}}d(x_i,f^i(x))\le\epsilon.
\]
\end{itemize}
Given a subset $S$ of $X$, we say that $f$ has {\em shadowing on $S$} if for any $\epsilon>0$, there is $\delta>0$ such that every $\delta$-pseudo orbit $(x_i)_{i\in\mathbb{Z}}$ of $f$ with $x_i\in S$ for all $i\in\mathbb{Z}$ is $\epsilon$-shadowed by some $x\in X$. We say that $f$ has {\em shadowing} if $f$ has shadowing on $X$.
\end{defi}

Similarly, we define expansiveness on a subset. 

\begin{defi}
\normalfont
Let $f\colon X\to X$ be a homeomorphism. For a subset $S$ of $X$, we say that $f$ is {\em expansive on $S$} if there is $e>0$ such that
\[
\sup_{i\in\mathbb{Z}}d(f^i(x),f^i(y))\le e
\]
implies $x=y$ for all $x,y\in\bigcap_{i\in\mathbb{Z}}f^i(S)$. In this case, $e>0$ is called an expansive constant for $f$ on $S$.
\end{defi}

Next, we recall the definition of topological entropy. Note that positive topological entropy is another characteristic feature of chaotic dynamical systems.

\begin{defi}
\normalfont
Let $f\colon X\to X$ be a continuous map and let $K$ be a subset of $X$. For $n\ge1$ and $r>0$, a subset $E$ of $K$ is said to be {\em $(n,r)$-separated} if $x\ne y$ implies
\[
\max_{0\le i\le n-1}d(f^i(x),f^i(y))>r
\]
for all $x,y\in E$. Let $s_n(f,K,r)$ denote the largest cardinality of an $(n,r)$-separated subset of $K$. We define $h(f,K,r)$ and $h(f,K)$ by
\[
h(f,K,r)=\limsup_{n\to\infty}\frac{1}{n}\log{s_n(f,K,r)}
\]
and $h(f,K)=\lim_{r\to0}h(f,K,r)$. The topological entropy $h_{\rm top}(f)$ of $f$ is defined by $h_{\rm top}(f)=h(f,X)$ (see, e.g., \cite{W} for more details).
\end{defi}

For $x\in X$ and a subset $S$ of $X$, let $d(x,S)$ denote the distance of $x$ from $S$:
\[
d(x,S)=\inf_{y\in S}d(x,y).
\]
For a subset $S$ of $X$ and $r>0$. we denote by $B_r(S)$ the closed $r$-neighborhood of $S$:
\[
B_r(S)=\{y\in X\colon d(x,S)\le r\}.
\]
For a homeomorphism $f\colon X\to X$, a closed subset $S$ of $X$ with $f(S)=S$ is said to be {\em locally maximal} if there is $r>0$ such that $S=\bigcap_{i\in\mathbb{Z}}f^i(B_r(S))$.

The main result of this paper is the following theorem. Through the shadowing lemma, the result applies directly to hyperbolic sets. 

\begin{thm}
Given a homeomorphism $f\colon X\to X$ and a closed subset $\Lambda$ of $X$ with $f(\Lambda)=\Lambda$, if
\begin{itemize}
\item $f$ has shadowing on $\Lambda$,
\item $f$ is expansive on $B_b(\Lambda)$ for some $b>0$,
\end{itemize}
then the following conditions are equivalent
\begin{itemize}
\item[(1)] $Sen(f|_{CR(f|_\Lambda)})=\emptyset$,
\item[(2)] $h_{\rm top}(f|_{\Lambda_\epsilon})=0$, where $\Lambda_\epsilon=\bigcap_{i\in\mathbb{Z}}f^i(B_\epsilon(\Lambda))$, for some $\epsilon>0$,
\item[(3)] for every $c>0$, there is a closed subset $\Gamma_c$ of $X$ with $f(\Gamma_c)=\Gamma_c$ such that
\begin{itemize}
\item $\Gamma_c$ is locally maximal,
\item $\Lambda\subset\Gamma_c\subset B_c(\Lambda)$ and $h_{\rm top}(f|_{\Gamma_c})=0$.
\end{itemize}
\end{itemize}
\end{thm}

Let $M$ be a closed Riemannian manifold and let $f\colon M\to M$ be a $C^1$-diffeomorphism. A closed subset $\Lambda$ of $M$ with $f(\Lambda)=\Lambda$ is called a {\em hyperbolic set} for $f$ if there are $C>0$, $0<\lambda<1$, and a $Df$-invariant splitting
\[
T_\Lambda M=E^u\oplus E^s
\]
such that for any $x\in\Lambda$,
\begin{itemize}
\item $||Df_x^n(v)||\le C\lambda^n||v||$ for all $v\in E_x^s$ and $n\ge0$,
\item $||Df_x^{-n}(v)||\le C\lambda^n||v||$ for all $v\in E_x^u$ and $n\ge0$.
\end{itemize}

\begin{rem}
\normalfont
We know that if $\Lambda$ is a hyperbolic set for $f$, then $\bigcap_{i\in\mathbb{Z}}f^i(B_b(\Lambda))$ is a hyperbolic set for $f$ for some $b>0$ (see \cite[Proposition 6.4.6]{KH}).
\end{rem}

The following is a consequence of the so-called shadowing lemma  (see, e.g., Theorem 18.1.2 and also Corollary 6.4.10 of \cite{KH}). 

\begin{lem}
Let $M$ be a closed Riemannian manifold and let
\[
f\colon M\to M
\]
be a $C^1$-diffeomorphism.  Given any hyperbolic set $\Lambda$ for $f$, $f$ is expansive and has shadowing on $B_b(\Lambda)$ for some $b>0$.
\end{lem}

As a corollary of Theorem 1.2 and Lemma 1.3, we obtain the following theorem, which characterizes non-chaotic (or, conversely, chaotic) hyperbolic sets. 

\begin{thm}
Let $M$ be a closed Riemannian manifold and let
\[
f\colon M\to M
\]
be a $C^1$-diffeomorphism. For a hyperbolic set  $\Lambda$ for $f$, the following conditions are equivalent
\begin{itemize}
\item[(1)] $Sen(f|_{CR(f|_\Lambda)})=\emptyset$,
\item[(2)] $h_{\rm top}(f|_{\Lambda_\epsilon})=0$, where $\Lambda_\epsilon=\bigcap_{i\in\mathbb{Z}}f^i(B_\epsilon(\Lambda))$, for some $\epsilon>0$,
\item[(3)] for every $c>0$, there is a locally maximal hyperbolic set $\Gamma_c$ for $f$ such that $\Lambda\subset\Gamma_c\subset B_c(\Lambda)$ and $h_{\rm top}(f|_{\Gamma_c})=0$.
\end{itemize}
\end{thm}

This paper consists of two sections and an appendix. In Section 2, we prove Theorem 1.2. In Appendix A, we give an alternative proof of a statement concerning Lemma 1.2.

\section{Proof of Theorem 1.2}

In this section, we prove Theorem 1.2. First, by adapting an argument in \cite{KM}, we prove the following lemma.

\begin{lem}
Given a homeomorphism $f\colon X\to X$ and a closed subset $S$ of $X$ with $f(S)=S$, if
\begin{itemize}
\item $S=CR(f|_S)$,
\item $Sen(f|_S)\ne\emptyset$,
\item $f$ has shadowing on $S$,
\end{itemize}
then for any $\epsilon>0$, there are $m\ge1$ and a closed subset $Y$ of $X$ with $f^m(Y)=Y$ such that
\begin{itemize}
\item $Y\subset\bigcap_{i\in\mathbb{Z}}f^i(B_\epsilon(S))$,
\item there is a surjective continuous map $\pi\colon Y\to\{0,1\}^\mathbb{Z}$ such that
\[
\pi\circ f^m|_Y=\sigma\circ\pi
\]
where $\sigma\colon\{0,1\}^\mathbb{Z}\to\{0,1\}^\mathbb{Z}$ is the shift map.
\end{itemize}
\end{lem}

\begin{proof}
Let $g=f|_S\colon S\to S$. Since $Sen(g)\ne\emptyset$, we have $Sen_a(g)\ne\emptyset$ for some $a>0$. Let $0<\epsilon<a/2$. Since $f$ has shadowing on $S$, there is $\delta>0$ such that every $\delta$-pseudo orbit $(v_i)_{i\in\mathbb{Z}}$ of $f$ with $v_i\in S$ for all $i\in\mathbb{Z}$ is $\epsilon$-shadowed by some $v\in X$. Let $p\in Sen_a(g)$ and $\eta>0$. As $p\in Sen_a(g)$, there are $x,y\in S$ and $k\ge1$
such that
\[
\max\{d(p,x),d(p,y)\}\le\eta
\]
and
\[
d(f^k(x),f^k(y))>a.
\]
Since $x,y\in S$ and $S=CR(g)$, there is a pair
\[
((x_i)_{i=0}^m,(y_i)_{i=0}^m)
\]
of $\eta$-chains of $g$ such that $(x_0,y_0)=(x_m,y_m)=(x,y)$ and $k+1\le m$. Let $(z_0,w_0)=(z_m,w_m)=(p,p)$ and
\[
(z_i,w_i)=(x_i,y_i)
\]
for all $1\le i\le m-1$. If $\eta>0$ is sufficiently small, then
\[
((z_i)_{i=0}^m,(w_i)_{i=0}^m)
\]
is a pair of $\delta$-chains of $g$ such that
\begin{align*}
d(z_k,w_k)-2\epsilon&\ge d(f^k(x),f^k(y))-2\epsilon-d(f^k(x),z_k)-d(f^k(y),w_k)\\
&>a-2\epsilon-d(f^k(x),x_k)-d(f^k(y),y_k)>0. 
\end{align*}
Given $s=(s_j)_{j\in\mathbb{Z}}\in\{z,w\}^\mathbb{Z}$, we define a $\delta$-pseudo orbit $\xi_s=(v_i^s)_{i\in\mathbb{Z}}$ of $f$ by
\[
v_{jm+r}^s=(s_j)_r
\]
for all $j\in\mathbb{Z}$ and $0\le r\le m-1$. Note that $v_i^s\in S$ for all $i\in\mathbb{Z}$. We consider a subset $Y$ of $X$ defined as
\[
Y=\{v\in X\colon\text{$\xi_s$ is $\epsilon$-shadowed by $v$ for some $s\in\{z,w\}^\mathbb{Z}$}\}.
\]
Note that $Y\subset\bigcap_{i\in\mathbb{Z}}f^i(B_\epsilon(S))$. We easily see that $Y$ is a closed subset of $X$ with $f^m(Y)=Y$. We define a map $\pi\colon Y\to\{z,w\}^\mathbb{Z}$ so that $\xi_{\pi(v)}$ is $\epsilon$-shadowed by $v$ for all $v\in Y$. It follows that $\pi$ is surjective, continuous, and satisfies
\[
\pi\circ f^m|_{Y}=\sigma\circ\pi
\]
where $\sigma\colon\{z,w\}^\mathbb{Z}\to\{z,w\}^\mathbb{Z}$ is the shift map. This completes the proof of the lemma.
\end{proof}

\begin{rem}
\normalfont
If $e>0$ is an expansive constant for $f$ on $B_b(S)$, where $b>0$, and
\[
0<\epsilon\le\min\{b,e/2\},
\]
then the map $\pi\colon Y\to\{z,w\}^\mathbb{Z}$ obtained in the above proof is a homeomorphism.
\end{rem}

Let $f\colon X\to X$ be a continuous map. For any $\epsilon>0$, there is $\delta>0$ such that if $(x_i)_{i=0}^k$ is a $\delta$-chain of $f$ with $x_0=x_k$, then $\{x_i\colon0\le i\le k\}\subset B_\epsilon(CR(f))$. By this, we easily see that
\[
CR(f)=CR(f|_{CR(f)}).
\]
Given a homeomorphism $f\colon X\to X$, let $\Lambda$ be a closed subset of $X$ with $f(\Lambda)=\Lambda$. Letting $S=CR(f|_\Lambda)$, we see that
\begin{itemize}
\item $f(S)=f|_\Lambda(S)=S$,
\item $S=CR(f|_\Lambda)=CR((f|_\Lambda)|_{CR(f|_\Lambda)})=CR((f|_\Lambda)|_S)=CR(f|_S)$,
\item if $f$ has shadowing on $\Lambda$, then $f$ has shadowing on $S$.
\end{itemize}

By Lemma 2.1 and the above observations, we obtain the following lemma, concerning the implication (2)$\implies$(1) in Theorem 1.2.

\begin{lem}
Given a homeomorphism $f\colon X\to X$ and a closed subset $\Lambda$ of $X$ with $f(\Lambda)=\Lambda$, if
\begin{itemize}
\item $Sen(f|_{CR(f|_\Lambda)})\ne\emptyset$,
\item $f$ has shadowing on $\Lambda$,
\end{itemize}
then for any $\epsilon>0$, there are $m\ge1$ and a closed subset $Y$ of $X$ with $f^m(Y)=Y$ such that
\begin{itemize}
\item $Y\subset\bigcap_{i\in\mathbb{Z}}f^i(B_\epsilon(\Lambda))$,
\item there is a surjective continuous map $\pi\colon Y\to\{0,1\}^\mathbb{Z}$ such that
\[
\pi\circ f^m|_Y=\sigma\circ\pi
\]
where $\sigma\colon\{0,1\}^\mathbb{Z}\to\{0,1\}^\mathbb{Z}$ is the shift map.
\end{itemize}
\end{lem}

\begin{rem}
\normalfont
The condition in Lemma 2.2 implies that
\[
h_{\rm top}(f|_{\Lambda_\epsilon})\ge h(f,Y)=\frac{1}{m}h(f^m,Y)\ge\frac{1}{m}h(\sigma,\{0,1\}^\mathbb{Z})=\frac{1}{m}\log2>0,
\]
where $\Lambda_\epsilon=\bigcap_{i\in\mathbb{Z}}f^i(B_\epsilon(\Lambda))$.
\end{rem}

The proof of the following lemma is not difficult, so left an exercise for the reader.

\begin{lem}
Let $f\colon X\to X$ be a homeomorphism and let $S$ be a closed subset of $X$ with $f(S)=S$. If
\begin{itemize}
\item $f|_S$ has shadowing,
\item $f$ is expansive on $B_b(S)$ for some $b>0$,
\end{itemize}
then $S$ is locally maximal.
\end{lem}

Given a continuous map $f\colon X\to X$, a closed subset $S$ of $X$ with $f(S)=S$ is called a {\em periodic orbit} for $f$ if there are $p\in S$ and $j\ge1$ such that $f^j(p)=p$ and $S=\{f^i(p)\colon0\le i\le j-1\}$. As a corollary of Lemma 2.3, we obtain the following lemma.

\begin{lem}
Let $f\colon X\to X$ be a homeomorphism. For a periodic orbit $S$ for $f$, if $f$ is expansive on $B_b(S)$ for some $b>0$, then $S$ is locally maximal.
\end{lem}

By using Lemma 2.4 and Theorem 1.1, we prove the following lemma, concerning the implication (1)$\implies$(2) in Theorem 1.2.

\begin{lem}
Given a homeomorphism $f\colon X\to X$ and a closed subset $\Lambda$ of $X$ with $f(\Lambda)=\Lambda$, if
\begin{itemize}
\item $f$ is expansive on $B_b(\Lambda)$ for some $b>0$,
\item $h_{\rm top}(f|_{\Lambda_\epsilon})>0$, where $\Lambda_\epsilon=\bigcap_{i\in\mathbb{Z}}f^i(B_\epsilon(\Lambda))$, for all $\epsilon>0$, 
\end{itemize}
then $Sen(f|_{CR(f|_\Lambda)})\ne\emptyset$.
\end{lem}

\begin{proof}
Let $0<\epsilon\le b$ and note that
\[
\Lambda_\epsilon\subset B_\epsilon(\Lambda)\subset B_b(\Lambda).
\]
Since $h_{\rm top}(f|_{\Lambda_\epsilon})>0$, there is $C_\epsilon\in\mathcal{C}(f|_{\Lambda_\epsilon})$ such that $C_\epsilon$ is an infinite set. We can take a sequence $b\ge\epsilon_1>\epsilon_2>\cdots$ and a closed subset $B$ of $X$ such that $\lim_{j\to\infty}\epsilon_j=0$ and $\lim_{j\to\infty}C_{\epsilon_j}=B$ (with respect to the Hausdorff distance). Note that $B\subset\Lambda$. Since $f(C_{\epsilon_j})=C_{\epsilon_j}$ for all $j\ge1$, we have $f(B)=B$. Since $f|_{C_{\epsilon_j}}$ is chain transitive for all $j\ge1$, $f|_B$ is chain transitive. We obtain $B\subset C$ for some $C\in\mathcal{C}(f|_\Lambda)$. If $C$ is a finite set, then $C$ is a periodic orbit for $f$ and so is $B$. Since $f$ is expansive on $B_b(B)$, Lemma 2.4 implies that $B$ is locally maximal. By $\lim_{j\to\infty}C_{\epsilon_j}=B$, we obtain $C_{\epsilon_j}=B$ for some $j\ge1$; however, this contradicts that $C_{\epsilon_j}$ is an infinite set. It follows that $C$ is an infinite set. Since $f|_\Lambda$ is expansive, by Theorem 1.1, we conclude that $Sen(f|_{CR(f|_\Lambda)})\ne\emptyset$, completing the proof.
\end{proof}

As the following example shows, the assumption in Lemma 2.5 that $f$ is expansive on $B_b(\Lambda)$ for some $b>0$ cannot be removed.

\begin{ex}
\normalfont
Let $C$ be the Cantor ternary set in $[0,1]$ and let $\sigma\colon\{0,1\}^\mathbb{Z}\to\{0,1\}^\mathbb{Z}$ be the shift map. Let $\phi\colon C\to\{0,1\}^\mathbb{Z}$ be a homeomorphism and let
\[
g=\phi^{-1}\circ\sigma\circ\phi\colon C\to C.
\]
We define a compact subset $X$ of $[0,1]^2$ and a homeomorphism $f\colon X\to X$ by
\begin{itemize}
\item
\[
X=\{(0,0)\}\cup\bigcup_{n\ge1}\{(n^{-1},n^{-1}y)\colon y\in C\},
\]
\item $f(0,0)=(0,0)$,
\item $f(n^{-1},n^{-1}y)=(n^{-1},n^{-1}g(y))$ for all $n\ge1$ and $y\in C$.
\end{itemize}
Let $\Lambda=\{(0,0)\}$ and note that $f(\Lambda)=\Lambda$. We easily see that
\begin{itemize}
\item $f|_\Lambda$ is (trivially) expansive,
\item $h_{\rm top}(f|_{\Lambda_\epsilon})\ge\log 2>0$, where $\Lambda_\epsilon=\bigcap_{i\in\mathbb{Z}}f^i(B_\epsilon(\Lambda))$, for all $\epsilon>0$;
\end{itemize}
however, $Sen(f|_{CR(f|_\Lambda)})=Sen(f|_\Lambda)=\emptyset$. 
\end{ex}

The next lemma can be proved by a standard argument.

\begin{lem}
Let $f\colon X\to X$ be a homeomorphism. Let $S$ be a closed subset of $X$ and let $e>0$ be an expansive constant for $f$ on $S$. For every $\epsilon>0$, there is $n\ge0$ such that
\[
\sup_{-n\le i \le n}d(f^i(x),f^i(y))\le e
\]
implies $d(x,y)\le\epsilon$ for all $x,y\in\bigcap_{i=-n}^nf^i(S)$.
\end{lem}

The following lemma is essentially proved in \cite{A2}, but here we give a proof for completeness.

\begin{lem}
Let $f\colon X\to X$ be a homeomorphism and let $\Lambda$ be a closed subset of $X$ with $f(\Lambda)=\Lambda$ such that
\begin{itemize}
\item $f$ has shadowing on $\Lambda$,
\item $f$ is expansive on $B_b(\Lambda)$ for some $b>0$.
\end{itemize}
If $\Lambda$ is totally disconnected, then for every $c>0$, there is a closed subset $\Gamma_c$ of $X$ such that
\begin{itemize}
\item $f(\Gamma_c)=\Gamma_c$ and $\Lambda\subset\Gamma_c\subset B_c(\Lambda)$,
\item there are a subshift of finite type $\Xi$ and a homeomorphism $h\colon\Xi\to\Gamma_c$ such that $h\circ\sigma|_\Xi=f|_{\Gamma_c}\circ h$, where $\sigma$ is the shift map, thus $f|_{\Gamma_c}$ has shadowing.
\end{itemize}
\end{lem}
 
\begin{proof}
Let $e>0$ be an expansive constant for $f$ on $B_b(\Lambda)$. Since $\Lambda$ is totally disconnected, there are clopen subsets $A_0,A_1,\dots,A_{m-1}$ of $\Lambda$, $m\ge1$, such that
\begin{itemize}
\item $\Lambda=\bigsqcup_{k=0}^{m-1}A_k$, a disjoint union,
\item $\sup_{x,y\in A_k}d(x,y)<e$ for every $0\le k\le m-1$.
\end{itemize}
We fix $0<c\le b$ such that
\begin{itemize}
\item $B_c(A_k)\cap B_c(A_l)=\emptyset$ for all $0\le k<l\le m-1$,
\item $\sup_{x,y\in B_c(A_k)}d(x,y)\le e$ for every $0\le k\le m-1$.
\end{itemize}
Let $B_k=B_c(A_k)$ for all $0\le k\le m-1$. Let
\[
\Sigma=\{(s_j)_{j\in\mathbb{Z}}\in\{0,1,\dots,m-1\}^\mathbb{Z}\colon\bigcap_{j\in\mathbb{Z}}f^{-j}(A_{s_j})\ne\emptyset\},
\]
a subshift, and
\[
W=\{(w_j)_{j=-n}^n\in\{0,1,\dots,m-1\}^{2n+1}\colon(w_j)_{j=-n}^n=(s_j)_{j=-n}^n\:\:\text{for some}\:\:(s_j)_{j\in\mathbb{Z}}\in\Sigma\},
\]
where $n\ge1$. Consider
\[
\Xi=\{(s_j)_{j\in\mathbb{Z}}\in\{0,1,\dots,m-1\}^\mathbb{Z}\colon(s_{j+a})_{a=-n}^n\in W\:\:\text{for all}\:\:j\in\mathbb{Z}\}
\]
and note that
\begin{itemize}
\item $\Sigma\subset\Xi$,
\item $\Xi$ is a subshift of finite type.
\end{itemize}
We shall show that if $n$ is sufficiently large, then for any $(s_j)_{j\in\mathbb{Z}}\in\Xi$,
\[
\bigcap_{j\in\mathbb{Z}}f^{-j}(B_{s_j})
\]
is non-empty and so a singleton. Let $s=(s_j)_{j\in\mathbb{Z}}\in\Xi$. For each $j\in\mathbb{Z}$, as $(s_{j+a})_{a=-n}^n\in W$, we have
\[
x_{s_j}\in\bigcap_{a=-n}^nf^{-a}(A_{s_{j+a}})
\]
for some $x_{s_j}\in A_{s_j}$. Let $\xi_s=(x_{s_j})_{j\in\mathbb{Z}}$ and note that $x_{s_j}\in\Lambda$ for all $j\in\mathbb{Z}$. For every $j\in\mathbb{Z}$, since
\[
f(x_{s_j})\in\bigcap_{a=-n}^nf^{-(a-1)}(A_{s_{j+a}})=\bigcap_{a=-n-1}^{n-1}f^{-a}(A_{s_{j+1+a}})
\]
and
\[
x_{s_{j+1}}\in\bigcap_{a=-n}^nf^{-a}(A_{s_{j+1+a}}),
\]
we have
\[
\{f(x_{s_j}),x_{s_{j+1}}\}\subset\bigcap_{a=-n}^{n-1}f^{-a}(A_{s_{j+1+a}})
\]
and thus
\[
\sup_{-n\le a\le n-1}d(f^a(f(x_{s_j})),f^a(x_{s_{j+1}}))<e.
\]
Since $f$ has shadowing on $\Lambda$, there is $\delta>0$ such that every $\delta$-pseudo orbit $(y_j)_{j\in\mathbb{Z}}$ of $f$ with $y_j\in\Lambda$ for all $j\in\mathbb{Z}$ is $c$-shadowed by some $y\in X$. If $n$ is sufficiently large, then by Lemma 2.6, we obtain
\[
\sup_{j\in\mathbb{Z}}d(f(x_{s_j}),x_{s_{j+1}})\le\delta,
\]
that is, $\xi_s$ is a $\delta$-pseudo orbit of $f$ with $x_{s_j}\in\Lambda$ for all $j\in\mathbb{Z}$ and so $c$-shadowed by some $x_s\in X$. This implies
\[
x_s\in\bigcap_{j\in\mathbb{Z}}f^{-j}(B_c(x_{s_j}))\subset\bigcap_{j\in\mathbb{Z}}f^{-j}(B_c(A_{s_j}))=\bigcap_{j\in\mathbb{Z}}f^{-j}(B_{s_j});
\]
therefore,
\[
\bigcap_{j\in\mathbb{Z}}f^{-j}(B_{s_j})\ne\emptyset.
\]
As $s=(s_j)_{j\in\mathbb{Z}}\in\Xi$ is arbitrary, we can define a map $h\colon \Xi\to X$ by
\[
h(s)\in\bigcap_{j\in\mathbb{Z}}f^{-j}(B_{s_j})
\]
for all $s=(s_j)_{j\in\mathbb{Z}}\in\Xi$. It follows that
\begin{itemize}
\item $h$ is injective,
\item by Lemma 2.6, $h$ is continuous,
\item $h\circ\sigma|_\Xi=f\circ h$, where $\sigma$ is the shift map,
\item $\Lambda=h(\Sigma)\subset h(\Xi)\subset B_c(\Lambda)$.
\end{itemize}
Letting $\Gamma_c=h(\Xi)$, we see that
\begin{itemize}
\item $f(\Gamma_c)=\Gamma_c$,
\item $\Lambda\subset\Gamma_c\subset B_c(\Lambda)$,
\item $h\colon\Xi\to\Gamma_c$ is a homeomorphism such that
\[
h\circ\sigma|_\Xi=f|_{\Gamma_c}\circ h;
\]
therefore, $f|_{\Gamma_c}$ has shadowing.
\end{itemize}
Since $0<c\le b$ is arbitrary, this completes the proof of the lemma.
\end{proof}

Finally, we complete the proof of Theorem 1.2.

\begin{proof}[Proof of Theorem 1.2]
The implication (1)$\implies$(2) (resp.\:(2)$\implies$(1)) is a consequence of Lemma 2.5 (resp.\:Lemma 2.2). We shall prove the implication (2)$\implies$(3). Assume that $h_{\rm top}(f|_{\Lambda_\epsilon})=0$ for some $\epsilon>0$. By (2)$\implies$(1), we have $Sen(f|_{CR(f|_\Lambda)})=\emptyset$. Since $f|_\Lambda$ is expansive, Theorem 1.1 and Remark 1.1 imply that $\Lambda$ is a countable set and so totally disconnected. Let $0<c<\min\{b,\epsilon\}$. By Lemma 2.7, we obtain a closed subset $\Gamma_c$ of $X$ such that
\begin{itemize}
\item $f(\Gamma_c)=\Gamma_c$ and $\Lambda\subset\Gamma_c\subset B_c(\Lambda)$,
\item $f|_{\Gamma_c}$ has shadowing.
\end{itemize}
Let $0<\eta\le b-c$. Since $f$ is expansive on $B_\eta(\Gamma_c)$ and $f|_{\Gamma_c}$ has shadowing, by Lemma 2.3, $\Gamma_c$ is locally maximal. By $\Gamma_c\subset B_c(\Lambda)\subset B_\epsilon(\Lambda)$, we obtain
\[
h_{\rm top}(f|_{\Gamma_c})\le h_{\rm top}(f|_{\Lambda_\epsilon})=0,
\]
i.e., $h_{\rm top}(f|_{\Gamma_c})=0$. As $\Lambda\subset\Gamma_c\subset B_c(\Lambda)$, the implication (2)$\implies$(3) has been proved. It remains to prove the  implication (3)$\implies$(2). Let $c>0$ and $\Gamma_c$ be as in $(3)$. Since $\Gamma_c$ is locally maximal, we have $\Gamma_c=\bigcap_{i\in\mathbb{Z}}f^i(B_\epsilon(\Gamma_c))$ for some $\epsilon>0$. By $B_\epsilon(\Lambda)\subset B_\epsilon(\Gamma_c)$, we obtain
\[
h_{\rm top}(f|_{\Lambda_\epsilon})\le h_{\rm top}(f|_{\Gamma_c})=0,
\]
i.e., $h_{\rm top}(f|_{\Lambda_\epsilon})=0$, thus the implication (3)$\implies$(2) has been proved. This completes the proof of the theorem.
\end{proof}

\appendix

\section{}

The aim of this appendix is to give an alternative proof of the following theorem.

\begin{thm}
For an expansive homeomorphism $f\colon X\to X$, if $X=Per(f)$, then $X$ is a finite set.
\end{thm}

The proof of Theorem A.1 is through the following lemma.

\begin{lem}
Let $f\colon X\to X$ be a homeomorphism. Let
\[
Y=\{x\in X\colon x\in\overline{X\setminus\{x\}}\}
\]
and note that $Sen(f)\subset Y$. If $X=Per(f)$, then $x\in\overline{Y\setminus\{x\}}$ for all $x\in Sen(f)$; therefore, $Y\ne\emptyset$ and $Y\subset Sen(f)$ imply that $Y$ is a perfect set.
\end{lem}

\begin{proof}[Proof of Lemma A.1]
Let $x\in Sen(f)$. Since $X=Per(f)$, we have $f^k(x)=x$ for some $k\ge1$. Let $g=f^k\colon X\to X$ and note that $g(x)=x$. As $x\in Sen(f)$, we have $x\in Sen(g)$ and so $x\in Sen_a(g)$ for some $a>0$. Let $\epsilon>0$. We take $0<\delta<a$ such that $d(x,w)\le\delta$ implies $d(x,g(w))\le\epsilon$ for all $w\in X$. Since $x\in Sen_a(g)$, there is a sequence $x_j\in X\setminus\{x\}$, $j\ge1$, such that
\begin{itemize}
\item $\sup_{j\ge1}d(x,x_j)\le\delta$ and $\lim_{j\to\infty}x_j=x$,
\item $\sup_{i\ge0}d(x,g^i(x_j))>a$ for all $j\ge1$.
\end{itemize}
Given $j\ge1$, let
\[
i_j=\min\{i\ge1\colon d(x,g^i(x_j))>\delta\}.
\]
Since $d(x,g^{i_j-1}(x_j))\le\delta$, we have $d(x,g^{i_j}(x_j))\le\epsilon$ and so
\[
\delta<d(x,g^{i_j}(x_j))\le\epsilon.
\]
By taking a subsequence if necessary, we may assume that $\lim_{j\to\infty}g^{i_j}(x_j)=y$ for some $y\in X$. Note that $\delta\le d(x,y)\le\epsilon$. Since $X=Per(f)$, we have $f^l(y)=y$ for some $l\ge1$. Let
\[
\mathcal{O}_f(y)=\{f^i(y)\colon 0\le i\le l-1\}.
\]
If $y\not\in Y$, then $y\not\in\overline{X\setminus\{y\}}$; therefore, $g^{i_j}(x_j)=y$ for all sufficiently large $j$. This implies $x_j\in\mathcal{O}_f(y)$ for all sufficiently large $j$, which is a contradiction. Thus, we obtain $y\in Y$. As $\delta\le d(x,y)\le\epsilon$ and $\epsilon>0$ is arbitrary, we obtain $x\in\overline{Y\setminus\{x\}}$. Since $x\in Sen(f)$ is arbitrary, the lemma has been proved.
\end{proof}

For the proof of Theorem A.1, we need a notation and a related fact. Let $f\colon X\to X$ be a homeomorphism. For $x\in X$ and $r>0$, we define a subset $W_r^s(x)$ of $X$ by
\[
W_r^s(x)=\{y\in X\colon\sup_{i\ge0}d(f^i(x),f^i(y))\le r\}.
\]
We know that if $e>0$ is an expansive constant for $f$, then $W_e^s(x)\subset W^s(x)$ for all $x\in X$.

\begin{proof}[Proof of Theorem A.1]
Since $f$ is expansive and $X=Per(f)$, $X$ is a countable set. Let $e>0$ be an expansive constant for $f$. For any $x\in X$, we have $W_e^s(x)\subset W^s(x)$ and so $W_e^s(x)=\{x\}$, because $X=Per(f)$. It follows that
\[
\sup_{i\ge0}d(f^i(x),f^i(y))>e
\]
for all $x,y\in X$ with $x\ne y$. Letting
\[
Y=\{x\in X\colon x\in\overline{X\setminus\{x\}}\},
\]
we obtain $Y=Sen_e(f)(\subset Sen(f))$. If $Y\ne\emptyset$, then by Lemma A.1, $Y$ is compact and perfect; therefore, $Y$ is an uncountable set, which is a contradiction. We conclude that $Y=\emptyset$ and thus $X$ is a finite set, completing the proof.
\end{proof}

As the following example shows, for a homeomorphism $f\colon X\to X$, $X=Per(f)$ does not necessarily imply $Sen(f)=\emptyset$.

\begin{ex}
\normalfont
Let $S^1=\{z\in\mathbb{C}\colon|z|=1\}$ and let
\[
z_{n,k}=\left(1-\frac{1}{n}\right)e^{2\pi i\cdot\frac{k}{2^{n-1}}}
\]
for all $n\ge1$ and $k\ge0$. Let
\[
X=S^1\cup\bigcup_{n\ge1}\left\{z_{n,k}\colon k\in\left\{0,1,\dots,2^{n-1}-1\right\}\right\}.
\]
We define a homeomorphism $f\colon X\to X$ by
\begin{itemize}
\item $f(z)=z$ for all $z\in S^1$,
\item $f(z_{n,k})=z_{n,k+1}$ for all $n\ge1$ and $k\in\left\{0,1,\dots,2^{n-1}-1\right\}$.
\end{itemize}
We easily see that $X=Per(f)$ and $S^1=Sen(f)$.
\end{ex}

\end{document}